\documentstyle[amssymb,amsfonts,12pt]{amsart}
\newtheorem{theorem}{Theorem}[section]
\newtheorem{lemma}[theorem]{Lemma}

\newtheorem{definition}[theorem]{Definition}

\theoremstyle{remark}
\newtheorem{remark}[theorem]{Remark}

\def\QSet{\mbox{\rm\kern.24em
\vrule width.03em height1.48ex depth-.051ex \kern-.26em Q}}

\def\E{{\mbox{\rm I\kern-.22em E}}}
\def\P{{\bf P}}

\def\T{{\bf T}}
\def\S{{\bf S}}
\def\Z{{\mathbb Z}}\def\P{{\mathbb P}}
\def\R{{\mathbb R}}

\def\C{{\mathbb C}}

\def\be#1{\begin{equation}\label{#1}}

\def\size{{\operatorname{size}}}
\def\mass{{\operatorname{mass}}}

\def\Y{{\mathcal R}}

\def\F{{\mathcal F}}

\def\I{{\mathcal I}}

\newcommand{\lb}{\langle}
\newcommand{\rb}{\rangle}

\def\bas{\begin{align*}}
\def\eas{\end{align*}}
\def\bi{\begin{itemize}}
\def\ei{\end{itemize}}
\newenvironment{proof}{\noindent {\bf Proof} }{\endprf\par}
\def \endprf{\hfill  {\vrule height6pt width6pt depth0pt}\medskip}
\def\emph#1{{\it #1}}

\begin{document}
\title[Some new light on a few classical results]{Some new light on a few classical results}
\author{Ciprian Demeter}
\address{Department of Mathematics, Indiana University, 831 East 3rd St., Bloomington IN 47405}
\email{demeterc@@indiana.edu}
\author{Prabath Silva}
\address{Department of Mathematics, Indiana University, 831 East 3rd St., Bloomington IN 47405}
\email{pssilva@@indiana.edu}

\thanks{The first author is supported by a Sloan Research Fellowship and by the NSF Grant DMS-1161752}

\begin{abstract}
The purpose of this paper is to describe a unified approach to proving vector-valued inequalities without relying on the full strength of weighted theory.  Our applications include  the Fefferman-Stein and Cordoba-Fefferman inequalities, as well as the vector-valued Carleson operator. Using this approach we also produce a proof of the boundedness of the classical bi-parameter multiplier operators, that does not rely on product theory. Our arguments are inspired by the vector valued restricted type interpolation used in \cite{BT}.

\end{abstract}
\maketitle

\section{The general principle}

In this paper we describe an alternative approach to a few well known vector-valued inequalities. One of them leads to an alternative way to estimate  bi-parameter linear operators. This approach has already played a crucial role in recent work in the linear setting \cite{BT} but also in the context of bilinear operators \cite{P1}, where weighted estimates were not available. At its core lies restricted type vector valued interpolation as encoded by the following principle:

\begin{theorem}[The general principle \cite{BT}]\label{GP}

Let $p_0, p_1\in (1,\infty)$ be such that $p_0 < p_1$ and let $\{T_j\}_j $ be a (possibly finite) sequence of sublinear operators on $\R^n$ which are uniformly bounded on $L^2$. Assume that for $p\in\{p_0,p_1\}$ there is $C_p>0$ with the following property:

$(P)$ for each finite nonzero measure sets $H,G\subset \R^n$ there exist subsets $H'\subset H$ and $G'\subset G$ with 
\begin{equation}
\label{nfvhfhbvuipofdigotgutiuhy}
|H'|\ge \frac12|H|,\;\;|G'|\ge \frac12|G|,
\end{equation}
such that
\begin{equation}\label{e2}
\int |T_j(f1_{H'})|^21_{G'}\le C_p\left(\frac{|G|}{|H|}\right)^{1-\frac2p}\int |f|^2
\end{equation}
for each $j$ and each $f\in L^2(\R^n).$

Then
\begin{equation}\label{vv}
\|(\sum_j|T_{j}f_j|)^{1/2}\|_q\lesssim_q \|(\sum_j|f_j|)^{1/2}\|_q
\end{equation}
for each $p_0<q<p_1$ and each $f_j$.
\end{theorem}

It is important to note that the choice of the subsets $H',G'$ as well as the constant $C_p$ are independent of $j$. In our applications we always can work with either $G'=G$ or $H'=H$ for a given value of $p$. However we believe that this more general form of the principle may one day find applications. 

The hypothesis that $\sup_j\|T_j\|_{L^2\to L^2}<\infty$ can be easily relaxed, but works fine with our applications.  

We use Theorem \ref{GP} to obtain vector-valued estimates for a family of operators by proving uniform $L^2$ estimates for a related family of operators. Indeed  if we define
\[S_{j,G',H'}(f)=T_j(f1_{H'})1_{G'},\]
then the estimate \eqref{e2} can be written as
 \[\|S_{j,G',H'}\|_{L^2\to L^2}\lesssim_p \left(\frac{|G|}{|H|}\right)^{\frac12-\frac1p}.\]

We prove Theorem \ref{GP} in Section \ref{secGP}. Note that having $L^p$ estimates for $T_j$ does not in general imply the above $L^2$ estimates for $S_{j,G',H'}$.

In Sections  \ref{s4}, \ref{s2} and  \ref{s3}  we give new proofs for three classical results using Theorem \ref{GP} and elements of the approach described in Section \ref{timefreq}. Sections \ref{s2} and \ref{s3} contain proofs for two classical bi-parameter problems: boundedness of bi-parameter multiplier operators and the Cordoba-Fefferman inequality. We reduce both these problems to vector-valued estimates for single scale operators and then use  Theorem \ref{GP} to prove these vector-valued estimates. The key advantage of this approach is that  we avoid product theory or explicit weighted theory and  reduce  bi-parameter problems to essentially single-parameter problems.

Our first application in Section \ref{s4} is a  proof of Fefferman-Stein inequality  that avoids explicit use of  weighted theory. This proof follows the line of argument from Section \ref{timefreq} in a much simpler setting. In Section \ref{vvC} we give a similar proof for the vector-valued estimates for the Carleson operator.

Our proofs are in general not easier than the classical ones. This is mostly due to technicalities associated with various decompositions. To keep the exposition as transparent as possible, we choose to focus mainly on how the General Principle \ref{GP} works in each case, and less on various other technicalities. We caution the reader that various parts of the argument need to be worked out in more detail and draw attention to the large body of literature where most of these details are explained in various related contexts.

Our hope is that the approach relying on Theorem \ref{GP} described in this paper will find further applications in the literature. We point out that the employment of this method was critical to the theorems proved in \cite{BT} and \cite{P1}.

We authors are grateful to Michael Bateman and Christoph Thiele for illuminating discussions.

\section{Proof of Theorem \ref{GP}}\label{secGP}
 A proof appears in \cite{BT}, but we include it here too,  for the reader's convenience.

Using generalized restricted type interpolation in the vector-valued setting, to obtain \eqref{vv} it is enough to show that the $l^2$ valued sublinear operator ${\bf T}$ defined by
$${\bf T}({\bf f})=(T_j(f_j))_j$$
for each ${\bf f}=(f_j)_{j}$ is restricted weak-type $(p,p)$ for $p\in\{p_1,p_2\}$. By that we mean that given any positive measure sets $G,H\subset \R^n$ we have
$$\int_G\|{\bf T}({\bf f})(x)\|_{l^2}dx\le A_p|H|^{1/p}|G|^{1/p'}$$
whenever $$\|{\bf f}(x)\|_{l^2}\le 1_{H}(x),\;\;a.e.\; x.$$

Note that to prove this for a fixed $p$ it suffices to prove the following superficially weaker statement: Let $\gamma=\gamma(p)=6^{\max(p,p')}$. Then given any positive measure sets $G,H\subset \R^n$ there exist subsets $H'\subset H$ and $G'\subset G$ with $|H'|\ge \frac{\gamma-1}{\gamma}|H|$ and $|G'|\ge \frac{\gamma-1}{\gamma}|G|$, such that
\begin{equation}
\label{kerjfjrgopithoy}
\int_{G'}\|{\bf T}({\bf f})(x)\|_{l^2}dx\le B_p|H|^{1/p}|G|^{1/p'}
\end{equation}
whenever $$\|{\bf f}(x)\|_{l^2}\le 1_{H'}(x),\;\;a.e.\; x.$$
Indeed, note first that
$$\int_{G}\|{\bf T}({\bf f})(x)\|_{l^2}dx\le \int_{G'}\|{\bf T}({\bf f}1_{H'})(x)\|_{l^2}dx+$$$$\int_{G\setminus G'}\|{\bf T}({\bf f}1_{H'})(x)\|_{l^2}dx+
\int_{G'}\|{\bf T}({\bf f}1_{H\setminus H'})(x)\|_{l^2}dx+\int_{G\setminus G'}\|{\bf T}({\bf f}1_{H\setminus H'})(x)\|_{l^2}dx.$$
The first term on the right hand side can be bounded by $B_p|H|^{1/p}|G|^{1/p'}$. For the remaining three terms  we iterate the decomposition.  Note that after $k$ iterations the error term is the sum of $3^k$ integrals of the form $\int_{G^*}\|{\bf T}({\bf f}1_{H^*})(x)\|_{l^2}dx$ with $|G^*||H^*|\le (\gamma^{-1})^k|G||H|$. Since the $T_j$ are uniformly bounded on $L^2$, each of these integrals is bounded by $C|G^*|^{1/2}|H^*|^{1/2}$. The choice of $\gamma$ forces the error term to go to zero.

By repeating the argument we are lead to the upper bound
$$\sum_{k=0}^{\infty}B_p\gamma^{-k\min(1/p,1/p')}3^k|H|^{1/p}|G|^{1/p'}\le A_p|H|^{1/p}|G|^{1/p'}.$$

Let now $G',H'$ be the subsets provided by the property $(P)$ in Theorem \ref{GP}. Using H\"{o}lder's  inequality, to verify \eqref{kerjfjrgopithoy} it is enough to show that
\begin{equation}\label{vv1}
\left\| \left( \sum_{j} |T_j(f_j)|^21_{G'} \right)^{1/2} \right\|_{2} \lesssim_p |H|^{\frac{1}{p}} |G|^{\frac{1}{p'}-\frac{1}{2}}.
\end{equation}
Then using the fact that $\{ f_j\}_j$ satisfy $\sum_{j} |f_j|^2 \le 1_{H'}$, in order to get \eqref{vv1} it is  enough to show that
\begin{equation}
\sum_j \| T_j(f_j)1_{G'}\|_2^2 \lesssim_p \left(\frac{|G|}{|H|}\right)^{1-\frac{2}{p}} \sum_j \|f_j\|_2^2.
\end{equation}
But this follows from \eqref{e2}.


\section{Some results from time frequency analysis}\label{timefreq}

In this section we briefly recall the main tools used in the proof of Carleson's theorem from \cite{LT}. These tools will be used in simpler settings to prove our results in the following sections. Let us start by recalling that the Carleson operator is defined by
\begin{equation}
\label{dfnu3rurt9387t845ugufhvrjiwfnh}
Cf(x)=\int_{\R} f(x+t)\frac{e^{iN(x)t}}{t}dt,
\end{equation}
where $N:\R\to\R$ is an arbitrary measurable choice function.

The approach developed in \cite{LT} relies on a size lemma, a mass lemma, a single tree estimate as well as on arguments involving obtaining better control over size and mass by removing exceptional sets.

We refer the reader to \cite{LT} and \cite{Th1} for the proofs of these results in the Fourier case. The discussion of the simpler Walsh case can be found in \cite{dem}.

A first crucial idea in \cite{LT} is to decompose the Carleson operator into discrete model operators, where wave packets are used to capture both frequency and spatial localizations.

\begin{definition}[Tiles and bi-tiles ]A tile $s$ is a product of two dyadic intervals with area $1$, that is  $s= I_s\times \omega_s$, with $|I_s| \times |\omega_s|=1$. The dyadic intervals  $I_s$ and $ \omega_s$ are respectively called the spatial interval and  the frequency  interval of $s$. A bi-tile $P=(P_1, P_2)$ is a pair of tiles $P_1,P_2$ with $I_{P_1}=I_{P_2}= I_P$ and such that  the intervals $\omega_{P_1}$ and $\omega_{P_2}$  have the same dyadic parent. We denote the frequency interval of the bi-tile $P$ by $\omega_P=\omega_{P_1} \cup \omega_{P_2}$.
\end{definition}
Given an finite interval $I\subset \R$ we denote its center by $c(I)$ and  by $ \tilde{\chi}_I $ the cutoff function given by
\begin{equation}\label{cutoff}
 \tilde{\chi}_I (x) = \left(1+ \left(\frac{x-c(I)}{|I|}\right)^2\right)^{-1/2}.
\end{equation}
\begin{definition}[Wave packet associated to a tile]\label{wavepacket}
Let $s=I_s \times \omega_s$ be a tile. A wave packet on $s$ is a smooth function $\varphi_s$ which has Fourier support in $\omega_s$ and obeys the spatial decay estimates
\[ |\frac{d^\alpha}{d x^{\alpha}}[e^{-ic(\omega_s)x}\varphi_s(x)]| \lesssim_{M,\alpha} |I_s|^{-\frac12-\alpha} \tilde{\chi}_{I_s}^M(x),\quad x\in\mathbb{R}, \]
for all $M>0$ and all non-negative integers $\alpha$.
\end{definition}

The following Fefferman ordering of bi-tiles is used in combinatorial arguments involving organizing the collection of bi-tiles.

\begin{definition}[Partial ordering bi-tiles] Given a pair of bi-tiles $P,P'$ we define $P<P'$ to mean
\[ I_P \subset I_{P'}, \ \ \ \omega_{P'} \subset \omega_P.\]
\end{definition}

\begin{remark}\label{2Dorder}
One may similarly define product tiles and wave packets in higher dimensions. The combinatorics in that context is much more difficult, in part due to the fact that there is no good substitute for the above partial ordering. To avoid this difficulty, in our forthcoming analysis of the  bi-parameter operators we use first use Littlewood-Paley theory to reduce matters to one dimensional vector-valued estimates.
\end{remark}
It turns out that the Carleson operator can be written as a superposition of discrete operators of the type
\begin{equation}
\label{erkjfr48fuioghjrtrfkwdlbi}
\sum_{P} \lb f, \varphi_{P_1} \rb\varphi_{P_2}(x) 1_{N^{-1}(\omega_{P_2})}(x).
\end{equation}

We will prefer to work with convex collections $\S$ of bi-tiles. That means that $P'\in\S$ whenever $P\le P'\le P''$ and $P,P''\in\S$.

Next we define a tree. This is a collection of tiles whose associate model sum plays the role of the Hilbert transform in the discrete setting.
\begin{definition}[Tree]\label{DefTree} A tree $T$ with top data $(\xi_T,I_T)$ is a convex collection of bi-tiles such that  for all $P\in T$ we have $I_P \subset I_T,$ and $\xi_T\in \omega_P.$
\end{definition}

To prove the boundedness of Carleson's operator it suffices to prove  uniform weak type bounds for the model operator in \eqref{erkjfr48fuioghjrtrfkwdlbi}, where one can restrict attention to finite,  convex collection of bi-tiles. Let us fix this collection and call it $\P_0$.

Next we recall the notion of {\em size}.

\begin{definition}[Size] Let $\S\subset \P_0$ be a collection of bi-tiles and let $f\in L^2(\R)$. The {\em size} of $\S$ with respect to $f$ is defined by
\[\size({\S}, f ) = \sup_{T} \left( \frac1{|I_T|}\sum_{P \in T} |\lb f , \varphi_{P_1} \rb |^2 \right)^{1/2},\]
where $\sup$  is taken over all the trees $T$ in $\S$ with top data $(\xi_T,I_T)$, that satisfy $\xi_T\in\omega_{P_2}$ for each $P\in\T$.
\end{definition}

The  following  lemma is used to partition a collection of bi-tiles into further subcollections with good control over size.

\begin{lemma}[Size lemma \cite{LT}]
\label{hggggh1}
Given a convex collection of bi-tiles $\S$ and $f\in L^2(\R)$ there exists a decomposition $\S =\S_{big}\cup \S_{small}$ such that $\S_{small}$ is convex,  $\size(\S_{small}) \le \frac{1}{2}\size (\S)$ and $\S_{big}= \bigcup_{T\in {\F}}T$, where ${\mathcal{F}}$ is a collection of trees (forest) with
\[\sum_{T \in {\F}} |I_T| \lesssim (\size(\S))^{-2} \|f\|_2.\]
\end{lemma}

Next we recall the concept of mass.  Given a measurable function $N:\R\to\R$, a bi-tile $P$ and a set $E \subset \R$, define
\[E_P=E \cap \{x: N(x) \in \omega_P\}.\]

\begin{definition}[Mass]  The mass of a convex collection of bi-tiles $\S$ is given by
\[\mass(\S,E) = \sup_{P \in {\S}} \frac{1}{|I_P|}\int_{E_{P}} \tilde{\chi}^{100}_{I_{P}}(x)dx,\]
\end{definition}

Similar to the size decomposition lemma, we have a mass decomposition lemma.

\begin{lemma}[Mass lemma \cite{LT}]
\label{hggggh2}
 Given a convex collection of bi-tiles $\S$ there exists a decomposition $\S=\S_{big}\cup \S_{small}, $ such that $\S_{small}$ is convex, $\mass(\S_{small}) \le \frac{1}{2}\mass(\S)$ and   $\S_{big}= \bigcup_{T\in \F}T$, where $\F$ is a collection of trees with
\[\sum_{T \in \F} |I_T| \lesssim (\mass(\S))^{-1} |E|.\]
\end{lemma}

The Mass lemma and Size lemma can be iterated to decompose a given convex collection of bi-tiles $\S$ as
\begin{equation}
\label{ejdghy834yt48937t8rjvio}
\S=\bigcup_{2^{-n}\le \size(\S)}\bigcup_{2^{-m}\le \mass(\S)}\S_{n,m},
\end{equation}
where $\S_{m,m}$ consists of a collection $\F_{n,m}$ of trees whose size and mass are bounded by $2^{-n}$ and $2^{-m}$ respectively and such that
$$\sum_{T\in\F_{n,m}}|I_T|\lesssim \min(2^{2n}\|f\|_2^2,2^{m}|E|).$$
We also recall
\begin{lemma}[Tree estimate \cite{LT}]\label{singletreeest} For a tree  $T$ we have the following estimate.
\[ \sum_{P\in T} |\lb f, \varphi_{P_1} \rb \lb 1_{E}, \varphi_{P_2} 1_{N^{-1}(\omega_{P_2})} \rb| \lesssim |I_T|  \size(T,f) \mass(T,E).\]
\end{lemma}

The final ingredient of the proof of the Carleson theorem is the argument involving removing exceptional sets to get better bounds for the  size and mass of a collection of bi-tiles. We need the following estimate on size. Note first that we have the trivial estimate $\size(\S,f) \lesssim \|f\|_\infty$.

\begin{lemma}[Size estimate]\label{sizeest}
If $\S$ is a convex collection of bi-tiles then we have
\[\size(\S,f) \lesssim\sup_{P\in \S}\inf_{x\in I_P} M(f)(x)\]
where $M$ is the Hardy-Littlewood maximal function.
\end{lemma}

When we remove exceptional sets to obtain a better bound for mass we use the following estimate. Note also that we always have the trivial estimate $\mass(\S) \lesssim 1$.

\begin{lemma}[Mass estimate]\label{densityest}
Let $\S$ be a collection of bi-tiles and $E\subset \R$. Then we have
\[\mass(\S) \lesssim\sup_{P\in \S}\inf_{x\in I_P} M(1_E)(x).\]
\end{lemma}

The argument of removing exceptional sets involves  decomposing  the collection of bi-tiles $\S$ into further sub-collections $\S_k, k \ge 0$ based on the position relative to  the exceptional sets and using Lemma \ref{sizeest} or Lemma \ref{densityest} to obtain better estimates for the size and mass of those subcollections. In the arguments from our paper we simply state the bounds we get for the subcollection $\S_0$, and  we refer the reader to page 16 of \cite{MTT2} for details on how to deal with $\S_k$, $k>0$.

The proof of the boundedness of Carleson's Theorem will follow by combining the decomposition \eqref{ejdghy834yt48937t8rjvio} with the Tree, Size and Mass estimate lemmas which become effective outside certain small exceptional sets. The result is a convergent double geometric sum. We refer the reader to Section 6 in \cite{dem} for the details.

\section{The Fefferman-Stein inequality}\label{s4}
In this section we give a proof for the Fefferman-Stein inequality in the dyadic case, using Theorem \ref{GP} and a very rudimentary version of the time-frequency tools recalled in the previous section. Let $$Mf(x):=\sup_{x\in I}\frac1{|I|}\int_I|f(y)|dy$$
be the dyadic maximal function, where $I$ runs over all dyadic intervals containing $x$.

\begin{theorem}[Fefferman-Stein, \cite{FS}]
For each $1<p<\infty$ and each $f_j$
$$\|(\sum_j|Mf_j|)^{1/2}\|_p\lesssim\|(\sum_j|f_j|)^{1/2}\|_p.$$
\end{theorem}

\begin{proof}
We first show the proof in the range $p>2$, where the classical argument relies on elementary weighted theory. We prove  \eqref{e2} for  fixed $G,H$. Since $p>2$ it suffices to consider the case $|G|\lesssim |H|$.  Define
$$H':=H\setminus(\bigcup_{I:\frac{|I\cap G|}{|I|}\ge c\frac{|G|}{|H|}} I),$$
for sufficiently large $c$, so that \eqref{nfvhfhbvuipofdigotgutiuhy} holds.
Fix an arbitrary  measurable
$$\kappa:\R\to \{2^{n}: n \in \Z\}.$$
For a dyadic interval $I$ define $V_I=\{x\in I:|I|=\kappa(x)\}$.
It suffices to check the General principle with
$$T_jf(x)=\sum_{I}\frac1{|I|}\langle f,1_I\rangle 1_{V_I}(x),$$
where the sum runs over all dyadic $I$.
Note that in this case all $T_j$ are the same.
We will prove \eqref{e2} for each $p>2$ using restricted interpolation. More precisely, we show that
$$\sum_{I}\frac1{|I|}|\langle 1_{E\cap H'},1_I\rangle \langle 1_{F\cap G},1_{V_I}\rangle|\lesssim (\frac{|G|}{|H|})^{1/s}|E|^{1/s}|F|^{1/s'},$$
for each $1<s<\infty$. Note that we can restrict the sum to the collection $\I$ of intervals $I$ which intersect $H'$.

In this case tiles are indexed by intervals and we have the following analogs of $\size$ and $\mass$,
$$\size(I):=\frac1{|I|}|\langle 1_{E\cap H'},1_I\rangle|,$$
$$\mass(I):=\frac1{|I|}|\langle 1_{F\cap G},1_{V_I}\rangle|.$$
Of course $\size(I)\lesssim 1$ and moreover $\mass(I)\lesssim \frac{|G|}{|H|}$ for each $I\in \I$. The latter inequality is due to the fact that $I\cap H'\not=\emptyset$, an instance of Lemma \ref{densityest}.

Let $\I_{n,m}^*$ be the collection of the maximal intervals in
$$\I_{n,m}:=\{I\in\I:\size(I)\sim 2^{-n}, \mass(I)\sim 2^{-m}\}.$$
For $J\in \I_{n,m}^*$ the collection $\{ I \in \I_{n,m}: I \subset J\}$ plays the role of a tree while the  collections $\I_{n,m}^*$ play the role of forests from Section \ref{timefreq}. Next we prove the following analogue of the single tree estimate, see Lemma \ref{singletreeest}. First note that for each $J\in \I_{n,m}^*$
$$\sum_{I\subset J:\atop{I\in \I_{n,m}}}\frac1{|I|}|\langle 1_{E\cap H'},1_I\rangle \langle 1_{F\cap G},1_{V_I}\rangle|=\int_\R\sum_{I\subset J\atop{I\in \I_{n,m}}}\frac1{|I|}\langle 1_{E\cap H'},1_I\rangle1_{F\cap G}1_{V_I}.$$
The support of the integral is actually a subset of $J\cap F\cap G$, and since each $x$ receives contribution from only one $I$, we have that the function we integrate is bounded by $2^{-n}$. Thus, the integral is bounded by
$$2^{-n}\mass(J)|J|\lesssim 2^{-n-m}|J|.$$
We easily get estimates similar to the ones in the Size Lemma and Mass Lemma from the previous section. Since each such $J \in  \I_{n,m}^*$ is a subset of $\{M1_E>2^{-n}\}$ and of $\{M1_F>2^{-m}\}$ we have
$$\sum_{J\in \I_{n,m}^*}|J|\lesssim \min\{2^{n}|E|,2^m|F|\}.$$
Note the improvement $2^n|E|$ versus $2^{2n}|E|$ in the Size Lemma. This comes from exploiting disjointness of supports ($L^1$ orthogonality) versus $L^2$ orthogonality.

Combining all these estimates  we get
\begin{align*}
&\sum_{2^{-n}\lesssim 1}\sum_{2^{-m}\lesssim \frac{|G|}{|H|}}\sum_{I\in \I_{n,m}}\frac1{|I|}|\langle 1_{E\cap H'},1_I\rangle \langle 1_{F\cap G},1_{V_I}\rangle| \\
&\lesssim \sum_{2^{-n}\lesssim 1}\sum_{2^{-m}\lesssim \frac{|G|}{|H|}} 2^{-n-m}\min\{2^{n}|E|,2^m|F|\}\lesssim (\frac{|G|}{|H|})^{1/s}|E|^{1/s}|F|^{1/s'},
\end{align*}
for each $1<s<\infty$. This shows that the operator $M_{H',G}(f):=M(f1_{H'})1_G$ satisfies restricted weak type bounds on $L^s$. Using the log convexity of the  implicit constants in restricted type interpolation, we immediately get that $\|M_{H',G}\|_{2\to 2}\lesssim (\frac{|G|}{|H|})^{1/2}$. This implies \eqref{e2} for each $p>2$ and as a result the Fefferman-Stein inequality follows in the $2<p<\infty$ range.

To get the $1<p<2$ range one has to repeat the above argument with
$$G':=G\setminus(\bigcup_{I:\frac{|H\cap I|}{|I|}\ge c\frac{|H|}{|G|}} I),$$
for sufficiently large $c$. We can of course assume $|H|\lesssim |G|$ in this case. By restricting attention to the intervals $I$ which intersect $G'$ we get an improved estimate for the size
$$\size(I)\lesssim \frac{|H|}{|G|}.$$ The previous computations will give $\|M_{H,G'}\|_{2\to 2}\lesssim (\frac{|H|}{|G|})^{1/2}$. This implies \eqref{e2} for each $p<2$ and as a result the Fefferman-Stein inequality follows in the $1<p<2$ range.
\end{proof}

\section{Bi-parameter multipliers}\label{s2}
Consider multipliers $m$ defined on $\R^2$ which satisfy the following bi-parameter H\"{o}rmander-Mihlin multiplier condition
$$|\partial^{\alpha}\partial^{\beta}m(\xi,\eta)|\lesssim\frac{1}{|\xi|^\alpha|\eta|^\beta},$$
for $\xi,\eta\not=0$ and sufficiently many $\alpha$, $\beta$.
It is known that the following  bilinear multiplier operator associated to $m$
$$Tf(x,y):=\int \widehat{f}(\xi,\eta)m(\xi,\eta)e^{2\pi i(x\xi+y\eta)}d\xi d\eta$$
is bounded on $L^p$, for $1<p<\infty$. The classical proofs rely on product BMO and product $H^1$. Here we give a different proof, whose only use of product theory is via the boundedness of the strong maximal function.

The boundedness of the operator $T$ can be reduced to that of model sums  of the form
\begin{equation}
\label{e3}
\sum_{R}\langle f,\phi_R\rangle\psi_R.
\end{equation}
the sum here is over all dyadic rectangles $R=I_R\times J_R\subset \R^2$. Here
\begin{equation}
\label{kfbrt78gy8945igukfgii}
 |\partial^{\alpha_1}_x\partial^{\alpha_2}_yF_R(x,y)| \lesssim_{M,\alpha} |I_R|^{-\frac12-\alpha_1}|J_R|^{-\frac12-\alpha_2} \tilde{\chi}_{I_R}^M(x)\tilde{\chi}_{J_R}^M(y)
\end{equation}
for all $M>0$  and sufficiently many non-negative integers $\alpha_i$, where  $F_R\in \{\phi_R,\psi_R\}$. Moreover, both $\widehat{\phi_R}$ and $\widehat{\psi_R}$ are supported in rectangles of the form $\omega_{R,1}\times \omega_{R,2}$ with $|\omega_{R,1}||I_R|\sim 1$, $|\omega_{R,2}||J_R|\sim 1$ and $\text{dist}(\omega_{R,i},0)\sim |\omega_{R,i}|$. See \cite{MPTT} for details.

We prove the following
\begin{theorem}
\label{t1}
For each $2<p<\infty$ and each $f_j$ we have
$$\|(\sum_{j\in\Z}|\sum_{R=I\times J:|J|=2^j}\langle f_j,\phi_R\rangle\psi_R|^2)^{1/2}\|_p\lesssim \|(\sum_j|f_j|^2)^{1/2}\|_p.$$
\end{theorem}
The boundedness of the model sums in \eqref{e3} for $p>2$ will follow by applying the above vector-valued inequality to $f_j:=S_j^2f$, where
$$\widehat{S_j^2f}(\xi,\eta)=\widehat{f}(\xi,\eta)1_{\omega_j}(\eta),$$
and $\text{dist}(\omega_j,0)\sim |\omega_j|\sim 2^{-j}$.
The boundedness for $1<p<2$ will follow by duality.

\begin{proof}[of Theorem \ref{t1}]
We apply Theorem \ref{GP} to
\[T_jf:=\sum_{R=I\times J:|J|=2^j}\langle f,\phi_R\rangle\psi_R.\]
Since the scale of $J$ is fixed, this is essentially a one-parameter multiplier and the bound $\|T_j\|_{p\to p}$ for each $1< p<\infty$ follows via classical one-dimensional theory. It remains to check \eqref{e2} for  $p\in(2,\infty)$.

Given $G$ and $H$, we note that \eqref{e2} with any $G'\subset G$ and $H'\subset H$ follows from the bound  $\|T_j\|_{p\to p}\lesssim 1$, in the case $|G|\gtrsim |H|$. Thus it suffices to assume $|G|\lesssim |H|$. In this case define
$$H':=H\setminus(\bigcup_{R:\frac{|R\cap G|}{|R|}>c_\epsilon(\frac{|G|}{|H|})^{1-\epsilon}}R),$$
where $\epsilon>0$ is small enough (it will depend on $p$) while $c_\epsilon$ is large enough so that \eqref{nfvhfhbvuipofdigotgutiuhy} holds. This can be achieved since the strong maximal function
$$M^*f(x,y):=\sup_{(x,y)\in R}\frac1{|R|}\int_R|f|,$$
maps $L^p$ to $L^p$, for $p>1$. Also note that the choice of the set $H'$ is independent of $j$, as desired.

Now consider the following operator
\begin{equation}
S_{j,G,H'}(f)=S_j(f) =\sum_{R=I_R\times J_R:|J_R|=2^j}\langle f 1_{H'},\phi_R\rangle\psi_R1_{G}.
\end{equation}
Recall that we have to prove that for each $\delta>0$
\begin{equation}\label{vv2}
\|S_{j,G,H'}(f)\|_2 \lesssim_\delta\left(\frac{|G|}{|H|}\right)^{\frac12-\delta}  \|f\|_2.
\end{equation}

Using the log convexity of the implicit constants in restricted type interpolation,  to prove \eqref{vv2} it is enough to show that for $E,F \subset \R$ with finite measure and functions $f,g$ with $|f|\le 1_{E}, |g|\le 1_{F}$,
\begin{equation}
\label{e4}
\sum_{R=I_R\times J_R:|J_R|=2^j}|\langle f1_{H'},\phi_R\rangle||\langle\psi_R, g1_{G}\rangle|\lesssim_{\epsilon,p} (\frac{|G|}{|H|})^{\frac{1-\epsilon}{p}}|E|^{1/p}|F|^{1/p'}.
\end{equation}
for each $2<p<\infty$. Indeed, it will suffice to interpolate this with the following consequence of the one dimensional type bound $\|T_j\|_{q\to q}\lesssim 1$ (use $q<2<p$ with $p$ much closer to 2 than $q$)
\begin{equation}\label{vvjrty3}
\sum_{R=I_R\times J_R:|J_R|=2^j}|\langle f1_{ H'},\phi_R\rangle||\langle\psi_R, g1_{G}\rangle|\lesssim_{q} |E|^{1/q}|F|^{1/q'}.
\end{equation}

The proof of \eqref{e4} follows a simpler version of the approach described in Section \ref{timefreq}. We will briefly sketch the details  here.

The nice feature of the operators $S_j$ is that they are essentially one-dimensional. In particular, the rectangles $\Y_j:=\{R:|J_R|=2^{j}\}$ are nicely ordered with respect to inclusion.
Note that since $\phi_R$ is mostly concentrated in $R$, the term $\langle f1_{H'},\phi_R\rangle$  will be small if $R\cap H'=\emptyset$.  This can be made precise, as described for example in \cite{MTT2}. To keep technicalities to a minimum we will focus only on the contribution coming from the collection $\Y_j$ of rectangles such that $R\cap H'\not=\emptyset$.

We have the following versions of the definitions and lemmas from Section \ref{timefreq}.

\begin{definition}
A collection $\Y\subset \Y_j$ is called {\em convex} if $R\subset R'\subset R''$ and $R,R''\in\Y$ imply that $R'\subset \Y$.
A {\em tree} $\T$ with top $R_\T$ is a convex collection of rectangles in $\Y_j$ such that $R\subset R_\T$ for each $R\in\T$.
\end{definition}
\begin{definition}
The {\em size} of a finite collection $\Y\subset \Y_j$ is defined by
$$\size(\Y):=\max_{\T\subset \Y}\left(\frac{1}{|R_\T|}\sum_{R\in\T}|\langle f1_{ H'}, \phi_R\rangle |^2\right)^{1/2},$$
where the maximum is taken over all the trees in $\Y$.
\end{definition}
\begin{definition}
The {\em mass} of  a convex collection $\Y\subset \Y_j$ is
$$\mass(\Y):=\max_{R\in \Y}\frac1{|R|}\int_{F\cap G}\tilde{\chi}_R^{100}.$$
\end{definition}

\begin{lemma} For each tree $\T$  we have
\begin{equation}
\label{e5}
\sum_{R\in\T}|\langle f1_{H'},\phi_R\rangle||\langle\psi_R, g1_{G}\rangle|\lesssim|R_\T|\size(\T)\mass(\T).
\end{equation}
\end{lemma}

By using limiting arguments we can and will assume that the sum in \eqref{e4} is over a finite convex  collection of rectangles $\Y$.

A key element of our construction of  $H'$ is that we have the following improvement over the trivial $O(1)$ bound on the mass
$$\mass(\Y_j)\lesssim_{\epsilon} (\frac{|G|}{|H|})^{1-\epsilon}.$$
The functions $\phi_R$ are almost orthogonal. As a consequence, Littlewood-Paley theory immediately implies that
$$\size(\Y_j)\lesssim 1.$$

Iterate Lemmas \ref{hggggh1} and \ref{hggggh2} (see again \cite{LT} for details)  to decompose
$$\Y=\bigcup_{2^{-n}\lesssim 1}\bigcup_{2^{-m}\lesssim_{\epsilon} (\frac{|G|}{|H|})^{1-\epsilon}}\Y_{n,m}, \ \ \ \text { with } \ \ \ \size(\Y_{n,m})\le 2^{-n}, \ \ \ \mass(\Y_{n,m})\le 2^{-m}$$
and each $\Y_{n,m}$ is the union of a family $\F_{n,m}$ of trees satisfying
$$\sum_{\T\in\F_{n,m}}|R_\T|\lesssim \min\{2^{2n}|E|,2^m|F|\}.$$
The final computations are as follows
\begin{align*}
\sum_{R\in \Y} |\langle f1_{H'},\phi_R\rangle||\langle\psi_R, g1_{G}\rangle|&= \sum_{2^{-n}\lesssim 1}\sum_{2^{-m}\lesssim_{\epsilon} (\frac{|G|}{|H|})^{1-\epsilon}}\sum_{\T\in\F_{n,m}}\sum_{R\in\T}|\langle f1_{H'},\phi_R\rangle||\langle\psi_R, g1_{G}\rangle|\\
&\lesssim \sum_{2^{-n}\lesssim 1}\sum_{2^{-m}\lesssim_{\epsilon} (\frac{|G|}{|H|})^{1-\epsilon}}\sum_{\T\in\F_{n,m}} |R_\T| 2^{-n}2^{-m}\\
&\lesssim \sum_{2^{-n}\lesssim 1}\sum_{2^{-m}\lesssim_{\epsilon} (\frac{|G|}{|H|})^{1-\epsilon}}2^{-n}2^{-m}(2^{2n}|E|)^{1/p}(2^m|F|)^{1/p'}\\ &\lesssim_{\epsilon,p} (\frac{|G|}{|H|})^{\frac{1-\epsilon}p}|E|^{1/p}|F|^{1/p'}.
\end{align*}
\end{proof}
A similar approach can extend the range in Theorem \ref{t1} to $1<p<\infty$, the details are left to the reader.

\section{The Cordoba-Fefferman inequality }\label{s3}

Define for a direction $v\in \R^2$ and $f:\R^2\to\C$
$$\widehat{H_vf}(\xi,\eta):=\hat{f}(\xi,\eta)1_{S_v}(\xi,\eta),$$
where $S_v$ is a half plane through the origin with normal vector $v$.

For any collection of directions $\Sigma:=\{v_j:j\in \Z\}\in \R^2$ define
$$M_\Sigma f(x,y)=\sup_{(x,y)\in R}\frac1{|R|}\int_R|f|,$$
where the supremum is taken over all rectangles $R$ containing $(x,y)$ and whose axes point in the directions $(v,v^\perp)$ with $v\in \Sigma$. In this section we reprove the following result due to Cordoba and Fefferman.
\begin{theorem}[\cite{CF}]
Consider a collection of directions $\Sigma:=\{v_j:j\in \Z\}$ such that $\|M_\Sigma\|_{L^p\to L^{p,\infty}}<\infty$ for some fixed $1<p<\infty$.
Then we have
\begin{equation}\label{CFest}
\|(\sum_j|H_{v_j}f_j|^2)^{1/2}\|_q\lesssim\|(\sum_j|f_j|^2)^{1/2}\|_q,
\end{equation}
for  $q$ in the range
$$\left|1-\frac2{q}\right|<\frac1p.$$ The implicit constant depends on $\|M_\Sigma\|_{L^p\to L^{p,\infty}}$.
\end{theorem}
\begin{proof}
To simplify a bit the exposition we work with the case $p=2$. The result is immediate when $q=2$. Using the fact that $H_v$ is self-dual and the fact that the dual of $L^q(l^2)$ is $L^{q'}(l^2)$, it will suffice to assume $2<q<4$.

Let $S_k$ be appropriate smooth annular truncations supported in $\{(\xi,\eta)\in \R^2:|(\xi,\eta)|\sim 2^k\}$, such that
$$\sum_{k\in \Z}S_kf=f.$$
In particular, by Littlewood-Paley theory
\begin{equation}
\label{e1}
\|(\sum_k|S_kf|^2)^{1/2}\|_q\sim \|f\|_q
\end{equation}
for each $1<q<\infty$.

We next remark that when $q>2$ we also have for arbitrary $f_j$
$$\|(\sum_j\sum_k|S_kf_j|^2)^{1/2}\|_q\sim \|(\sum_j|f_j|)^{1/2}\|_q.$$
This follows from a standard argument based on randomization with a doubly indexed Rademacher sequence. Indeed, on  one hand (see \cite{St}, Appendix D) we have
$$\int |\sum_k\sum_jr_k(\omega_1)r_j(\omega_2)S_kf_j(x)|^qdxd\omega_1d\omega_2\sim \int(\sum_{k,j}|S_kf_j(x)|^2)^{q/2}dx.$$
On the other hand, using \eqref{e1} we get
\begin{align*}
&\int |\sum_k\sum_jr_k(\omega_1)r_j(\omega_2)S_kf_j(x)|^qdxd\omega_1d\omega_2=\int |\sum_kr_k(\omega_1)S_k(\sum_jr_j(\omega_2)f_j)(x)|^qd\omega_1dxd\omega_2\\
 &\sim\int \left(\sum_k |S_k(\sum_jr_j(\omega_2)f_j)(x)|^2 \right)^{\frac{q}{2}} dx d\omega_2
\sim \int |\sum_jr_j(\omega_2)f_j(x)|^qd\omega_2dx\\
&\sim \int (\sum_j|f_j(x)|^2)^{q/2}dx.
\end{align*}

Next we prove that
$$\|(\sum_j\sum_k|T_{j,k}f_{j,k}|)^{1/2}\|_q\lesssim\|(\sum_j\sum_k|f_{j,k}|)^{1/2}\|_q,$$
where $T_{j,k}=S_kH_{v_j}$. Note that  \eqref{CFest} follows from this if we take $f_{j,k}:=S_k'f_j$, with $S_k'$ an appropriate modification of $S_k$.

We will apply the general principle to the operators $T_{j,k}$, with $p_0=2$ and $ p_1=4-\epsilon$. Note that the multiplier $m_{j,k}$ of $T_{j,k}$ satisfies\footnote{In reality the multiplier is only singular with respect to the $v$ axis, but to make the argument more symmetric we pretend it is also singular with respect to the $v^\perp$ axis}
$$|\partial^{\alpha}\partial^{\beta}m_{j,k}(Rot_{v_j}(\xi,\eta))|\lesssim \frac{1}{|\xi|^\alpha|\eta|^{\beta}},$$
where $Rot_v$ is the rotation around the origin which maps the $x$-axis to the line with direction $v$. As explained in Section \ref{s2}, we can assume
$$T_{j,k}f=\sum_{R\in\S_{j,k}}\langle f, \phi_R\rangle\psi_R,$$
where $\S_{j,k}$ consists of rectangles in the direction of $v_j$, with one side of fixed length $2^{-k}$. Also $\phi_R(Rot_{v_j}(x))$ and  $\psi_R(Rot_{v_j}(x))$ will satisfy the same properties as the functions $F_R$ from the previous section.

Let
$$L:=\|M_\Sigma\|_{L^2\to L^{2,\infty}}.$$
Given $H$ and $G$ such that $|G|\lesssim |H|$ define
$$H':=H\setminus\left(\bigcup_{j,k}\bigcup_{R\in \S_{j,k}:\frac{|R\cap G|}{|R|}\ge c(\frac{|G|}{|H|})^{1/2}L} R\right).$$
It is easy to see that \eqref{nfvhfhbvuipofdigotgutiuhy} is satisfied if $c$ is large enough.
Define $$C_{j,k}f:= T_{j,k}(f1_{H'})1_{G}.$$
We will prove \eqref{e2} with $p_1$ arbitrarily close to (but less than) 4. That is
$$\|C_{j,k}\|_{2\to 2}\lesssim (\frac{|G|}{|H|})^{\alpha}$$
for each $0<\alpha<1/4$. We rely on  restricted type interpolation.

The  argument described in Section \ref{s2} will apply here too. Fix $E,F \subset \R$ with finite measure and functions $f,g$ with $|f|\le 1_{E}, |g|\le 1_{F}$. We focus again only on  those $R\in\S_{j,k}$ which intersect $H'$.
Thus, for each $2<p<\infty$
$$\sum_{R\in\S_{j,k}}|\langle f1_{ H'},\phi_R\rangle||\langle\psi_R, g1_{G}\rangle|\lesssim \sum_{2^{-m}\lesssim (\frac{|G|}{|H|})^{1/2}L}\sum_{2^{-n}\lesssim 1}2^{-n}2^{-m}\min\{2^{m}|F|,2^{2n}|E|\}$$$$\lesssim|E|^{1/p}|F|^{1/p'}(\frac{|G|}{|H|})^{\frac1{2p}}L^{1/p.}$$
Using restricted type interpolation  we get the following  for each $2<p$,
$$\|C_{j,k}\|_{p\to p}\lesssim (\frac{|G|}{|H|})^{\frac1{2p}}L^{1/p}.$$

Interpolating the above  bound for $p$ very close to 2 with the easy one dimensional bound $\|C_{j,k}\|_{s \to s}\lesssim 1$ for a (any) fixed $s\in (1,2)$  one gets for each $\epsilon>0$ small enough
$$\|C_{j,k}\|_{2\to 2}\lesssim_\epsilon (\frac{|G|}{|H|})^{\frac14-\delta_1(\epsilon)}L^{\frac12-\delta_2(\epsilon)},$$
where $\delta_i(\epsilon)\to 0$ as $\epsilon\to 0$.

\end{proof}

The refinement from Section 6.8 in \cite{GR} of the proof of the Cordoba-Fefferman result, combined with the sharp estimate for the Hilbert transform in weighted spaces \cite{P}  proves the following stronger result. In particular it recovers the endpoints $|1-\frac2{q}|= \frac1p$, but the bound depends on the strong, rather than the weak $L^p$ norm of the maximal function. Since our  approach relies on interpolation, it does not recover this stronger form of the result. We present this argument for reader's convenience.
\begin{theorem}
Consider a collection of vectors $\Sigma:=\{v_j:j\in \Z\}$ such that $\|M_\Sigma\|_{L^p\to L^{p}}<\infty$ for some fixed $1<p<\infty$.
Then for each function $f_j$ and each $q$ so that
$$|1-\frac2{q}|\le \frac1p$$ we have
\begin{equation}\label{bsuibkjsdb}
\|(\sum_j|H_{v_j}f_j|^2)^{1/2}\|_q\lesssim\|M_\Sigma\|_{{p}\to {p}}^{p|1-\frac2{q}|}\|(\sum_j|f_j|^2)^{1/2}\|_q.
\end{equation}
\end{theorem}

\begin{proof}
First, recall a general result about weights $u$ on $\R$. Assume that $u$ is an $A_1$ weight, that is
$$Mu(x)\le \|u\|_{A_1}u(x),\;\;\;\;a.e.\;\;\;x.$$
Then  we have
\begin{align*}
\|u\|_{A_2}:&=\sup_{I\subset \R:I\text{ interval}}(\frac{1}{|I|}\int_Iu(x)dx)(\frac{1}{|I|}\int_Iu^{-1}(x)dx)\\
&\le \sup_{I\subset \R:I\text{ interval}}2\inf_{x\in I}Mu(x)\sup_{x\in I}\frac1{u(x)} \ \ \le 2\|u\|_{A_1}.
\end{align*}

Using this and the sharp result in \cite{P} it follows that
\begin{equation}
\label{dhveruyfjvjh}
\int |Hf|^2u\le C\|u\|_{A_1}^2\int|f|^2u,
\end{equation}
where $H$ is the Hilbert transform, and $C$ is independent of $f$ and $u$.

By duality it suffices to assume $q\ge 2$. Take $g\in L^p(\R^2)$. Define
$$w(x)=\sum_{k=0}^{\infty}\frac{1}{(2\|M_\Sigma\|_{L^{p}\to L^{p}})^k}M_\Sigma^kg(x),$$
where $M_\Sigma^k$ is the composition of $M_\Sigma$ with itself $k$ times. Note that we have,
$$M_\Sigma w(x)\le 2\|M_\Sigma\|_{L^{p}\to L^{p}}w(x), \ \ \ g(x)\le w(x), \ \ \ \|w\|_p\le 2\|g\|_p.$$
Using these and also \eqref{dhveruyfjvjh}  we get
$$\int |H_{v_j}f_j|^2g\le \int |H_{v_j}f_j|^2w\le 4C\|M_\Sigma\|_{L^{p}\to L^{p}}^2\int |f_j|^2w.$$
Next note that by interpolation it is enough to consider the endpoint $q=2p'$. This case follows from the inequalities below
$$\|(\sum_j|H_{v_j}f_j|^2)^{1/2}\|_{2p'}^2 = \|\sum_j|H_{v_j}f_j|^2\|_{p'} = \sup_{g\in L^p, \|g\|_p\le1 } \int \sum_j |H_{v_j}f_j|^2g $$
$$\lesssim  \|M_\Sigma\|_{L^{p}\to L^{p}}^2 \sup_{g\in L^p, \|g\|_p=1 } \int \sum_j |f_j|^2w \lesssim \|M_\Sigma\|_{L^{p}\to L^{p}}^2\|(\sum_j|f_j|^2)^{1/2}\|_{2p'}^2,$$
where in the last inequality  we have used the fact that $\| w\|_p \le 2\|g\|_p \le 2$.
\end{proof}

%
%
%
%
%
%
%
%
%
%

\section{Vector-valued estimates for the Carleson operator}\label{vvC}

In this section we sketch the proof of the vector-valued estimates for the Carleson operator defined in \eqref{dfnu3rurt9387t845ugufhvrjiwfnh}.
\begin{theorem}[\cite{FRT}]\label{vvcarloseon} Let $1<p<\infty$, then for each $f_j$
$$\|(\sum_j|Cf_j|)^{1/2}\|_p\lesssim\|(\sum_j|f_j|)^{1/2}\|_p.$$
\end{theorem}
The classical proof from \cite{FRT} relies on weighted estimates for the Carleson operator. Our approach relies on the fact that Carleson's operator is bounded and on standard refinements of the proof of its boundedness.
We again refer the reader to the  Section \ref{timefreq} for the relevant tools.

\begin{proof}For two sets $A,B$ the operator $S_{A,B}$ is defined by
\[S_{A,B}(f)=T(f1_B)1_A.\]
Here $T$ is the model sum operator in \eqref{erkjfr48fuioghjrtrfkwdlbi}.
First consider the case when $p>2$. Given sets $G, H$ with $|G|\lesssim |H|$ define $H'=H\setminus\{M1_{G} \gtrsim\frac{|G|}{|H|}\}$. It is enough to prove that  for $\epsilon>0$
\begin{equation}\label{Cest1}
\|S_{G,H'}\|_{L^2\to L^2}\lesssim_{\epsilon} (\frac{|G|}{|H|})^{1/2-\epsilon}.
\end{equation}
As before, to keep the argument as nontechnical as possible we only focus on the main contribution, the one coming from bi-tiles whose spatial interval intersect $H'$.

Let $t>2$. Fix  $|f| \le 1_E, |g| \le 1_F$. Then using the bound for the mass $2^{-m}\lesssim \frac{|G|}{|H|}$ guaranteed by the definition of $H'$ and Lemma \ref{densityest}, the trivial bound on the size $2^{-n}\lesssim 1$ and the machinery described in Section \ref{timefreq} we get
\begin{align*}
|\langle S_{G,H'}(f),g\rangle| &=|\langle T(f1_{H'}),g1_G\rangle|\lesssim
\sum_{2^{-n}\lesssim 1}\sum_{2^{-m}\lesssim \frac{|G|}{|H|}}2^{-n-m}(2^{m}|F|)^{\frac1{t'}}(2^{2n}|E|)^{\frac1t}\\
&\lesssim (\frac{|G|}{|H|})^{1/t}|E|^{1/t}|F|^{1/t'}.
\end{align*}
Interpolate this bound for $t>2$ very close to 2 with the classical $O(1)$ restricted bound below 2 for the Carleson operator, to get the desired estimate.

Assume now $p<2$. In this case we remove an exceptional set from $G$. Given $|G|\gtrsim |H|$ define $G'=G\setminus\{M1_{H}\gtrsim\frac{|H|}{|G|}\}$.  It will suffice to prove the operator norm
$$\|S_{G',H}\|_{L^2\to L^2}\lesssim_{\epsilon} (\frac{|H|}{|G|})^{\frac12-\epsilon}$$
for each $\epsilon>0$. We only focus on the bi-tiles whose spatial intervals intersect $G'$.

Let $t>2$. Fix  $|f| \le 1_E, |g| \le 1_F$. Then  using the bound for the size $2^{-n}\lesssim \frac{|H|}{|G|}$ guaranteed by the definition of $G'$ and Lemma \ref{sizeest}, the trivial bound on the mass $2^{-m}\lesssim 1$ and  the machinery described in Section \ref{timefreq} we get
\begin{align*}
|\langle S_{G',H}(f),g\rangle|=|\langle C(f1_{H}),1_{G'}g\rangle| &\lesssim  \sum_{2^{-m}\lesssim 1}\sum_{2^{-n}\lesssim \frac{|H|}{|G|}}2^{-n-m}(2^{m}|F|)^{\frac1{t'}}(2^{2n}|E|)^{\frac1t}\\&\lesssim (\frac{|H|}{|G|})^{1-\frac2t}|E|^{1/t}|F|^{1/t'}.
\end{align*}
Interpolate this bound for $t\to\infty$ with the classical $O(1)$ restricted bound for $T$  near $L^1$  to get again the desired estimate.
\end{proof}

\end{document}